\documentclass[12pt]{amsart}

\newcommand\real{{\mathrm I}\!{\mathrm R} }

\newcommand\nat{{{\mathrm I}\!\!\!{\mathrm N}}}
\newcommand\rat{{\mathrm Q}\kern-.65em {}^{{}_/ }}

\newtheorem{corollary}{Corollary}

\newtheorem{theorem}{Theorem}
\newtheorem{lemma}{Lemma}

\begin{document}
\title{Equilibrium states for smooth maps}
\author{Abdelhamid Amroun}
\maketitle
\address{\begin{center}
Universit\'e Paris-Sud, D\'epartement de Math\'ematiques,
CNRS UMR 8628, 91405 Orsay Cedex France
\end{center}}

\begin{abstract} We prove an equidistribution result for $C^{\infty}$
maps with respect to equilibrium states. 
We apply the result to the time-one map of
the geodesic flow of a closed smooth Riemannian manifold.
\end{abstract}

\section{Introduction and main results}

The techniques of large deviations are widely used in the theory of
dynamical systems to describe their statistical properties (see for
example \cite{Ki}, \cite{lsy} and \cite{lsyg}). In this paper we will use
these techniques to prove an equidistribution result for
smooth maps (Theorem $1$ $(3)$) with
respect to equilibrium states, i.e invariant
measures which maximize the topological pressure. Furthermore, we
will prove that the proportion of orbits supporting a Dirac measure close to an
equilibrium state is close to $1$ (Theorem $1$ $(2)$).
The general method combines the ideas in \cite{amr} for the geodesic
flow, where a more geometric point of view was considered. 
The main result (Theorem $1$) applies to the
time-one map of the geodesic flow of compact Riemannian manifold
and as a consequence we give a new version of the results in
\cite{amr} (Corollary $1$).

The process for which the large
deviations are computed lies in the space of probability measures of
the phase space. 
Let $X$ be a compact metric space and $\mathcal{P}(X)$ the space of probability
measures on $X$ endowed with the topology of weak convergence of measures.
Let $\{\nu_{n}\}_{n\in \nat}$ be a family of probability measures on
$\mathcal{P}(X)$. A large deviation principle for this
process consists on the two following bounds: for any closed
subset $K$ and any open subset $O$ of $\mathcal{P}(X)$,
\begin{equation}
\limsup_{n\rightarrow \infty}\frac{1}{n}\log \nu_{n}(K)\leq -J(K), \ 
\liminf_{n\rightarrow \infty}\frac{1}{n}\log \nu_{n}(O)\geq -J(O)
\end{equation}
where $J$ is some positive functional defined on subsets of
$\mathcal{P}(X)$. 
In what follows, we will give a precise description of $(1)$ for
smooth maps on a compact manifold $X$ (Theorem $1$ $(1)$ and $(4)$).
Under some conditions we will
prove a contraction principle for $\{\nu_{n}\}_{n\in \nat}$ which is a
large deviation principle with constraints (Theorem $2$).
While a direct proof is given for the upper bound in $(1)$, the lower
bound in $(1)$ will follow from this contraction principle. In other
words, to obtain the lower bound, we have to prove a finite
dimensional version of the large deviation lower bound
(see Theorem $2$). The main tool in this work is a formula by
Kozlovski \cite{Koz} for the topological pressure of a
$C^{\infty}$-map (Theorem $3$ below). This formula suggests that, at
least when there is a unique equilibrium state, the convergence will
take place exponentially fast, but it is not clear for what reasonable
class of potentials. Certainly we have to impose some hyperbolic
structure on the dynamics but it is remarkable that
Kozlovski's formula only requires smoothness.

\subsection{Preliminaries and notations}
Let $X$ be a smooth compact  manifold and $f: X\rightarrow X$ a 
$C^{\infty}$ map. We assume that $X$ has volume one, $\int_{X}dx=1$. 
We denote by $\mathcal{P}(X)$ the space of probability measures on $X$
endowed with the weak star topology. 
Let $\mathcal{P}_{inv}(X)$ be the
subset of $\mathcal{P}(X)$ of $f$-invariant probability measures. Given a
potential $\gamma \in C_{\real}(X)$, the topological pressure of
$\gamma$ is the 
number defined by the variational principle \cite{Wa},
\begin{equation}
P(\gamma)=\sup_{m\in \mathcal{P}_{inv}(X)}
\left ( h(m)+\int_{X} \gamma dm \right ),
\end{equation}
where $h(m)$ is the entropy of the measure $m$. For $F=0$ this reduces to
$P(0)=\sup_{m\in \mathcal{P}_{inv}(X)}h(m):=h_{top}$,
where $h_{top}$ is the topological entropy of $f$. 
An equilibrium state for $\gamma$ is a measure $m\in
\mathcal{P}_{inv}(X)$ which achieves the maximum in $(2)$:
\[
h(m)+\int_{X} \gamma dm=P(\gamma).
\]
We denote by $\mathcal{P}_{e}(\gamma)$ the subset of
$\mathcal{P}_{inv}(X)$ of equilibrium states corresponding to $\gamma$. By
a result of Newhouse \cite{Ne}, since $f$ is $C^{\infty}$, the
entropy map $m\rightarrow h(m)$ is upper semicontinuous. Then
$h_{top}<\infty$  
and consequently, the set $\mathcal{P}_{e}(\gamma)$
is a nonempty closed, compact, convex subset of $\mathcal{P}(X)$ \cite{Wa}.
 
We define the functional $Q_{\gamma}$ on $C_{\real}(X)$ based on the
potential $\gamma$ by,
\begin{equation}
Q_{\gamma}(\omega):=P(\gamma+\omega)-P(\gamma).
\end{equation}
By definition, $Q_{\gamma}$ is continuous
on continuous functions (see Lemma $1$).
Sometimes we will simply write $Q$, if there is no confusion to be
been afraid. We set for any probability measure $\mu$ on $X$,
\begin{equation}
J_{\gamma}(\mu):=\sup_{\omega}(\int \omega d\mu- Q_{\gamma}(\omega)),
\end{equation}
where the $\sup$ is taken over the space of continuous functions $\omega$ on
$X$. Observe that since $Q_{\gamma}(0)=0$, then $J_{\gamma}$ is a non
negative functional and clearly is lower semicontinuous.
We will see that (Lemma $1$) that $\mu \in \mathcal{P}_{e}(\gamma)$ if
and only if $\mu$ is invariant and $J_{\gamma}(\mu)=0$. 
Again, if there is no ambiguity we write $J$ instead of $J_{\gamma}$.
Since $J$ is convex and lower semicontinuous, we have by duality,
\begin{equation}
Q_{\gamma}(\omega)=
\sup_{\mu \in \mathcal{P}(X)}(\int \omega d\mu- J_{\gamma}(\mu)).
\end{equation}
For any set $E\subset \mathcal{P}(X)$ put
\[
J_{\gamma}(E):=\inf_{\mu \in E}J_{\gamma}(\mu).
\]

\subsection{The results}
We do the following assumption under which we
prove the lower bound part of Theorem $1$ (see also
\cite{Ki} and \cite{DZ}).

{\it Assumption A. There exists a countable set $\mathcal{C}$ of
 functions $\{g_{k}, k\geq 1\}\subset C_{\real}(X)$ such that their
 span is dense in 
 $C_{\real}(X)$ with respect to the topology of uniforme convergence,
  $\|g_{k}\|=1$ for all $k$, and for all $\beta \in \real^{n}$ the
  potential $\sum_{k=1}^{n}\beta_{k}g_{k}$ has a unique equilibrium state.}

This assumption holds in particular for H\"older continuous functions
when the dynamical system is strongly hyperbolic. This is the
case for Anosov flows (in particular for the geodesic flow of
negatively curved compact manifold) and hyperbolic diffeomorphisms
\cite{ru} \cite{fr}.

Given $x\in X$ we set
$\delta_{n}(x):=\frac{1}{n}\sum_{i=0}^{n-1}\delta_{f^{i}(x)}$, where
$\delta_{y}$ is the Dirac measure at $y$ and,
$\int_{X}\omega d\delta_{n}(x)
:=\frac{1}{n}\sum_{i=0}^{n-1}\omega(f^{i}(x))$. Sometimes, to simplify
the expressions we will simply denote it by $\delta_{n}(x)(\omega)$.

We also define the measures $l_{n}(dx):=\|\wedge (D_{x}f^{n})\|dx$ for
each $n\geq 1$, where $\|\wedge(D_{x}f^{n})\|$ can be seen as the maximum of
the volume of the images under $D_{x}f^{n}$ of the arbitrary
dimensional cubes with volume $1$. More formally,
$\|\wedge (D_{x}f^{n})\|:=\max_{j\leq dim X}\|\wedge^{j} (D_{x}f^{n})\|,
$ with
\[
\|\wedge^{j} (D_{x}f^{n})\|=\max_{V\in Gr(dim X, j)}|det
(D_{x}f^{n}|V)|,
\] 
where $Gr(dim X, j)$ (the grassmannian) is the set of all the subspace
of $T_{x}X$ of dimension $j$. 

\begin{theorem} Let $f:X\rightarrow X$ be a
$\mathcal{C}^{\infty}$ map of a smooth compact manifold $X$.
Then, for any continuous function $\gamma \in C_{\real}(X)$ we have:
\begin{enumerate}
\item For any closed subset $K$ of $\mathcal{P}(X)$
\[
\limsup_{n\rightarrow \infty}\frac{1}{n}\log
\frac{\int_{X} e^{n\int \gamma d\delta_{n}(x)}
1_{(\delta_{n}(x) \in K)}l_{n}(dx)}
{\int_{X} e^{n\int \gamma d\delta_{n}(x)}l_{n}(dx)}
\leq -J(K).
\]
\item For any open neighborhood $V$ of $\mathcal{P}_{e}(\gamma)$ we
  have,
\[
\lim_{n\rightarrow \infty}
\frac{\int_{X} e^{n\int \gamma d\delta_{n}(x)}
1_{(\delta_{n}(x) \in V)}l_{n}(dx)}
{\int_{X} e^{n\int \gamma d\delta_{n}(x)}l_{n}(dx)}=1,
\]where the convergence is exponential with speed $e^{-nJ(V^{c})}$.
\item The weak limits of
\[
\mu_{n}:= \frac{\int_{X} 
e^{n\int \gamma d\delta_{n}(x)}\delta_{n}(x)l_{n}(dx)}
{\int_{X}e^{n\int \gamma d\delta_{n}(x)}l_{n}(dx)}
\]
are equilibrium states corresponding to the potential $\gamma$,
i.e any weak limit $\mu_{\infty}$ is an invariant probability measure
and satisfies 
\[
h(\mu_{\infty})+\int_{X}\gamma d\mu_{\infty}=P(\gamma).
\]
\item Assume A. If for all $\beta \in \real^{d}$ and $g=(g_{1},\ldots,
  g_{d})\in \mathcal{C}^{d}$, $\gamma+\beta \cdot g$ has a unique 
equilibrium state, then for any open subset $O$ of $\mathcal{P}(X)$
\[
\liminf_{n\rightarrow \infty}\frac{1}{n}\log
\frac{\int_{X} e^{n\int \gamma d\delta_{n}(x)}
1_{(\delta_{n}(x) \in O)}l_{n}(dx)}
{\int_{X} e^{n\int \gamma d\delta_{n}(x)}l_{n}(dx)}
\geq -J(O).
\]
\end{enumerate}
\end{theorem}
The upper bound $(1)$ can be interpreted as a measurement of the growth of the
$l_{n}$-volume of certain sets as follows. Let us write it for $\gamma =0$ for
simplicity: for $n$ sufficiently large we get from $(1)$,
\[
e^{-nl_{n}(X)} l_{n}(\{x\in X: \delta_{n}(x) \in K\})\leq 
e^{-nJ_{0}(K)}
\]
where $J_{0}$ is the functional $J_{\gamma}$ corresponding to $\gamma
=0$. Recall from $(4)$ that in this case we have for any probability
measure $m$ on $X$, $J_{0}(m)=\sup_{\omega}(\int_{X}\omega dm- P(\omega))+
h_{top}$. Furthermore, if $K$ contains an invariant probability
measure $\mu_{K}$ and $J_{0}$ achieves its minimum in $K$ at $\mu$,
$J_{0}(K)=J_{0}(\mu_{K})$, then by Lemma $1$ hereunder we will have
$J_{0}(\mu)=h_{top}-h(\mu_{K})$. Thus for $n$ sufficiently large we get,
\[
e^{-nl_{n}(X)}l_{n}(\{x\in X: \delta_{n}(x) \in K\})\leq
e^{-n(h_{top}-h(\mu_{K})}.
\]

Point ($3$) of this theorem will be deduced from the point ($1$)
and there is no need of {\it Assumption A}.  

Theorem $1$ apply in particular to the time-one map of the geodesic
flow of the unit tangent bundle of a
compact Riemannian manifold $M$. In this case, the weak limits of $\mu_{n}$ 
are equilibrium states of the geodesic flow
corresponding to the potential $\gamma$. 
If $M$ is a manifold of negative curvature, it is well known that for
any H\"older potential there exists a unique equilibrium
state, then the corresponding measures $\mu_{n}$ converge to this
state. There are three well known invariant measures in this 
setting. The Bowen-Margulis measure $\mu_{0}$, which is the equilibrium state
(a measure of maximal entropy) corresponding to the constant
potential. The harmonic measure $\nu$ which
corresponds to the potential $\frac{d}{dt}|_{t=0}(K\circ
\widetilde{\varphi}_{t})$  
where $K$ is the Poisson kernel and $\widetilde{\varphi}_{t}$ is the
geodesic flow of $S\widetilde{M}$. The  Liouville measure $m_{liou}$
which is the 
equilibrium state of the potential $\frac{d}{dt}|_{t=0} \det \left(
d\varphi_{t} |_{E^{s}}\right )$ where $E^{s}$ is the stable tangent
bundle of $SM$ (see \cite{ch} and \cite{che} for more details).
If $M$ is a rank 1 manifold (Riemannian manifolds of nonpositive
curvature), Knieper \cite{Kn} showed that there exists
a uniquely determined invariant measure of maximal entropy
$\mu_{\max}$ for the geodesic flow and then $\mu_{n}$ converges
towards $\mu_{\max}$. As a consequence of Theorem $1$ we deduce the
following result for the geodesic flow.

\begin{corollary} Let $M$ be a closed and connected manifold equipped
  with a $C^{\infty}$ Riemannian metric and $f\equiv \varphi_{1}$ the
 time-one map of the geodesic flow $\phi$ of the unit tangent bundle
 $X=SM$ of $M$. Let $\gamma \in C_{\real}(SM)$ be a continuous potential.
Then the weak limits of 
\[
\mu_{n}:= 
\frac{\int_{SM} e^{n\int \gamma d\delta_{n}(\theta)}
  \delta_{n}(\theta)l_{n}(d \theta)}
{\int_{SM} e^{n\int \gamma d\delta_{n}(\theta)} l_{n}(d \theta)}
\]
are equilibrium states for the
geodesic flow corresponding to the potential $\gamma$. Here $d\theta$
is the volume form induced by the Riemannian metric on the tangent
bundle $TM$ and $l_{n}(d \theta):=
\|\wedge(D_{\theta}\varphi_{1}^{n})\| d\theta$. 
\end{corollary}

\subsection{Contraction Principle}
For $g\in C_{\real^{d}}(X)$ and $\alpha \in \real^{d}$ we set
$m(g):=\int gdm$ and
\[
\mathcal{P}_{g,\alpha}(X):=\{m\in \mathcal{P}(X):
m(g)=\alpha \}.
\]
We define the functionals:
\begin{equation*}
J_{g}(\alpha)=\left\{
          \begin{array}{ll}\inf (J(m):m\in \mathcal{P}_{g,\alpha}(X))&
  \quad \mathrm{if} \quad \mathcal{P}_{g,\alpha}(X)\ne
  \emptyset \\ 
 +\infty & \quad \mathrm{if} \quad
  \mathcal{P}_{g,\alpha}(X)=\emptyset \\          
          \end{array}
        \right.
\end{equation*}
and $J_{g}(E_{d})=\inf \left ( J_{g}(\alpha): \alpha \in E_{d}\right )$ 
for any $E_{d} \subset \real^{d}$. The following result is known as a
contraction principle \cite{DZ}. 
\begin{theorem}[Contraction Principle]
Let $f:X\rightarrow X$ a
$\mathcal{C}^{\infty}$ map of a smooth compact manifold $X$.
Let $\gamma :X \rightarrow\real$ be a continuous
potential and $g\in C_{\real^{d}}(X)$. Then,
\begin{enumerate}
\item For any closed subset $K_{d} \subset \real^{d}$,
\[
\limsup_{n\rightarrow \infty}\frac{1}{n}\log
\frac{\int_{X}e^{n\int \gamma d\delta_{n}(x)}
1_{(\delta_{n}(x)(g) \in K_{d})}l_{n}(dx)}
{\int_{X}e^{n\int \gamma d\delta_{n}(x)}l_{n}(dx)}
\leq -J_{g}(K_{d}).
\]
\item  If for all $\beta \in \real^{d}$, $\gamma+\beta \cdot g$ 
has a unique equilibrium state, then for any open subset 
$O_{d}\subset \real^{d}$, 
\[
\liminf_{n\rightarrow \infty}\frac{1}{n}\log
\frac{\int_{X}e^{n\int \gamma d\delta_{n}(x)}
1_{(\delta_{n}(x)(g) \in O_{d})}l_{n}(dx)}
{\int_{X}e^{n\int \gamma d\delta_{n}(x)}l_{n}(dx)}
\geq -J_{g}(O_{d}).
\]
\end{enumerate}
\end{theorem}

\section{Proofs}
\subsection{Preliminaries}
For any invariant probability measure $\mu$ we set
\begin{equation}
I(\mu):=P(\gamma)-(h(\mu)+\int_{SM} \gamma d\mu).
\end{equation}
\begin{lemma} 
\begin{enumerate}
\item $Q_{\gamma}$ is $f$-invariant, that is
$Q_{\gamma}(\omega \circ f)=Q_{\gamma}(\omega)$ for all
continuous function $\omega$. Moreover,
$Q_{\gamma}$ is convex and continuous on continuous functions.  
\item $Q_{\gamma}(\omega)=\sup_{\mu \in \mathcal{P}_{inv}(SM)}(\int \omega
  d\mu-I(\mu))$. In other words, the functionals $I$ and $J_{\gamma}$
agree on invariant measures.
\item $\mathcal{P}_{e}(\gamma)=\{J_{\gamma}=0\}$, that is $J_{\gamma}$
vanishes on a probability measure $\mu$ if and only if $\mu$ is an
equilibrium state for $f$.
\end{enumerate}
\end{lemma}
\begin{proof}
The fact that $Q_{\gamma}$ is $f$-invariant follows from its
definition $(3)$.
Part $(2)$ is a consequence of the convexity of the pressure function
$P$ and the variational principle $(1)$ from which we can
easily deduce that $|P(f)-P(g)|\leq \|f-g\|_{\infty}$ \cite{Wa}. 
Part $(2)$ follows from $(7)$ and,
\begin{eqnarray*}
&&\sup_{\mu \in \mathcal{P}_{inv}(X)}(\int \omega  d\mu-I(\mu))\\
&=& \sup_{\mu \in \mathcal{P}_{inv}(X)}(\int \omega  d\mu-
P(\gamma )+h(\mu)+\int \gamma d\mu)\\
&=&P(\gamma+\omega)-P(\gamma)=Q_{\gamma}(\omega).
\end{eqnarray*}
Thus, a probability measure $m$
satisfies $J_{\gamma}(m)=0$ if and only if $m$ is invariant and
$I(m)=0$. 
On the other hand, by definition and the continuity of $Q_{\gamma}$, the
functional $J_{\gamma}$ is lower semicontinuous. 
Thus, if $K$ is a closed subset of $\mathcal{P}(X)$ we have
$\inf_{m\in K}J_{\gamma}(m)=J_{\gamma}(\mu)$ for some $\mu \in K$. 
Then $\inf_{m\in K}J_{\gamma}(m)=0$ iff $\mu \in K$ is invariant and
$h(\mu)+\int 
\gamma d\mu=P(\gamma)$. In other words, by Lemma $1$ $(2)$, we have
$\mathcal{P}_{e}(\gamma)=\{J_{\gamma}=0\}$. 
\end{proof}
The proofs below are based on a formula of Kozlovski \cite{Koz} which
asserts that the topological pressure for a $\mathcal{C}^{\infty}$ map
of a smooth compact manifold is given by the exponential growth of the
mapping $\wedge(D_{x}f^{n})$ between full exterior algebras of the
tangent spaces.   
\begin{theorem}[O. S. Kozlovski] Let $f:X\rightarrow X$ be a
$\mathcal{C}^{\infty}$ map of a smooth compact 
manifold $X$. Then, the
topological pressure $P(\gamma)$ of a potential $\gamma \in
C_{\real}(X)$ is given by  
\begin{equation}
P(\gamma)=\lim_{n\rightarrow \infty}\frac{1}{n}\log \int_{X}
e^{n\int \gamma d\delta_{n}(x)}l_{n}(dx).
\end{equation}
\end{theorem}

\subsection{Proof of Theorem 1 ($1$)}
\begin{proof}
Set for any subset $E$ of $\mathcal{P}(X)$,
\begin{equation}
\nu_{n}(E):=
\frac{\int_{X}e^{n\int \gamma d\delta_{n}(x)}
1_{(\delta_{n}(x) \in E)}l_{n}(dx)}
{\int_{X} e^{n\int \gamma d\delta_{n}(x)}l_{n}(dx)}.
\end{equation}
We have to prove
\[
\limsup_{n\rightarrow \infty}\frac{1}{n}\log \nu_{n}(K) \leq
-\inf_{m\in K}J(m):=-J(K)
\]for any closed subset $K$.
Let $\epsilon >0$. Observe that the set $K$ is contained the union of
open sets,
\[
K\subset \cup_{\omega}\{\mu \in \mathcal{P}(X) :\int
\omega dm-Q(\omega)>J(K)-\epsilon\}.
\]
There exists then a finite number of continuous
functions $\omega_{1}, \cdots, \omega_{l}$ such that $K\subset
\cup_{i=1}^{l}K_{i}$, where 
\[
K_{i}=\{m\in \mathcal{P}(X): \int \omega_{i}dm-Q(\omega_{i})>J(K)-\epsilon\}.
\]
Put
\[
\Gamma_{i}(x,n):=\{x\in X:\delta_{n}(x)\in K_{i}\}, \ and \
Z_{i}(n):=\int_{\Gamma_{i}(x,n)} 
e^{n\int \gamma d\delta_{n}(x)}l_{n}(dx)
\]
We have $\nu_{n}(K) \leq \sum_{i=1}^{l}\nu_{n}(K_{i})$, where
$\nu_{n}(K_{i})=\frac{Z_{i}(n)}
{\int_{X} e^{n\int \gamma d\delta_{n}(x)}l_{n}(dx)}$.
By definition of $Q$ and $\Gamma_{i}(x,n)$, and Theorem $3$ we have for
$n$ sufficiently large,
\begin{eqnarray*}
&&\nu_{n}(K)\\
&\leq&\sum_{i=1}^{l}Z_{i}(n)e^{-n(P(\gamma)-\epsilon)}\\
&\leq&e^{-n(P(\gamma)-\epsilon)} \sum_{i=1}^{l}\int_{\Gamma_{i}(x,n)}
e^{n\int \gamma d\delta_{n}(x)}
e^{n(\int \omega_{i}d\delta_{n}(x)-
Q(\omega_{i})-(J(K)-\epsilon))}l_{n}(dx)\\
&\leq&e^{n(-Q(\omega_{i})-(J(K)-\epsilon))}
e^{-n(P(\gamma)-\epsilon)}\sum_{i=1}^{l}\int_{\Gamma_{i}(x,n)}
e^{n\int (\gamma +\omega_{i}) d\delta_{n}(x)}l_{n}(dx)\\
&\leq&\sum_{i=1}^{l}e^{-n(P(\gamma)-\epsilon)}
e^{n(-Q(\omega_{i})-(J(K)-\epsilon))}
e^{n(P(\gamma +\omega_{i})+\epsilon)}\\
&=& le^{n(-J(K)+3\epsilon)}.
\end{eqnarray*}
Take the logarithme, divide by $n$ and the $\limsup$,
\[
\limsup_{n\rightarrow \infty}\frac{1}{n}\log \nu_{n}(K) 
\leq -J(K)+3\epsilon.
\]
$\epsilon$ being arbitrary, this proves Theorem $1$ $(1)$.
\end{proof}

\subsection{Proof of Theorem $1$ ($2$)}
\begin{proof}
It is a consequence of Theorem $1$ $(1)$. Indeed, we have to prove that
$\lim_{n\rightarrow \infty}\nu_{n}(V)=1$. Set $K=V^{c}$ the complement
of $V$ in $\mathcal{P}(X)$. We have $J(K)=J(m)$ for some $m\in
K$. Thus by Lemma $1$, $J(K)>0$ and for $n$ sufficiently large,
\[
1\geq \nu_{n}(V)\geq 1-e^{-nJ(K)}.
\]
\end{proof}
\subsection{Proof of Theorem $1$ ($3$)}
\begin{proof}
We have to show that the weak limits of
\[
\mu_{n}:= \frac{\int_{X} 
e^{n\int \gamma d\delta_{n}(x)}\delta_{n}(x)l_{n}(dx)}
{\int_{X}e^{n\int \gamma d\delta_{n}(x)}l_{n}(dx)}
\]
are contained in $\mathcal{P}_{e}(\gamma)$. Observe that we can write
$\mu_{n}=E_{\beta_{n}}(\delta_{n})$ where the expectation is taken
with respect to the probability measure on $X$,
\[
\beta_{n}(E):=\frac{\int_{E} 
e^{n\int \gamma d\delta_{n}(x)}l_{n}(dx)}
{\int_{X}e^{n\int \gamma d\delta_{n}(x)}l_{n}(dx)}.
\]
This means that we view $\delta_{n}$ as
a random variable on the probability space $(X,\beta_{n})$.

Let $V\subset \mathcal{P}(X)$ be a convex open 
neighborhood of $\mathcal{P}_{e}(\gamma)$ and $\epsilon >0$.
We consider a finite open cover $(B_{i}(\epsilon))_{i\leq N}$ of
$\mathcal{P}_{e}(\gamma)$ by balls of diameter $\epsilon$ all
contained in $V$.  

Decompose the set $U:=\cup_{i=1}^{N}B_{i}(\epsilon)$ as follows,
\[
U=\cup_{j=1}^{N'}U_{j}^{\epsilon},\ N'\geq N,
\]
where the sets $U_{j}^{\epsilon}$ are disjoints and contained in one
of the balls $(B_{i}(\epsilon))_{i\leq N}$. We have
\[
\mathcal{P}_{e}(\gamma) \subset U \subset V,
\] 
and $\sum_{j=1}^{N'}\nu_{n}(U_{j}^{\epsilon})=\nu_{n}(U)$.
We fix in each $U_{j}^{\epsilon}$ a probability measure $p_{j}$,
$j\leq N'$, and let $p_{0}$ be a probability  measure distinct from the above
ones (for example take $p_{0} \in V\backslash U$).

Define the following process on the space $(X,\beta_{n})$,
\[
\omega_{n}:=\sum_{j=1}^{N'}p_{j}1_{\delta_{n}^{-1}(U_{j}^{\epsilon})}+
(1-\nu_{n}(U))p_{0}.
\]
We have,
\begin{equation}
E_{\beta_{n}}(\omega_{n})=
\sum_{j=1}^{N'}\nu_{n}(U_{j}^{\epsilon})p_{j}+
(1-\nu_{n}(U))p_{0}. 
\end{equation}
Since $span \{g_{1}, g_{2}, \cdots\}$ is dense in $C_{\real}(X)$,
the topology generated by the metric:
\[
d(m, m'):=\sum_{k=1}^{\infty}2^{-k}|\int g_{k}dm - \int g_{k}dm' |,
\]  
is compatible with weak star topology. We will use it to evaluate the
distance between elements in $\mathcal{P}(X)$.

The probability measure $E_{\beta_{n}}(\omega_{n})$ lies in $V$ since
it is a convex combination of elements of the convex set $V$. 
We have then
$d(\mu_{n}, V)\leq d(\mu_{n},E_{\beta_{n}}(\omega_{n}))$. We
will show that
\[
d(\mu_{n},E_{\beta_{n}}(\omega_{n}))\leq \epsilon
\nu_{n}(U)+\frac{3}{2}\nu_{n}(U^{c}), 
\]
where $U^{c}=\mathcal{P}(X)\backslash U$ which is closed.

Set $\mu_{n,V}:=E_{\beta_{n}}\left ((1_{V}\circ
\delta_{n})\delta_{n} \right )$ (this defines a finite measure on $X$).
By definition of $\mu_{n}$ and $\mu_{n,V}$ and
the fact that $U\subset V$, we get
\[
\sum_{k\geq 1}2^{-k}|\mu_{n}(g_{k}) - \mu_{n,V}(g_{k})| \leq
\frac{1}{2}\nu_{n}(U^{c}). 
\]
It remains to show that 
\[
\sum_{k\geq 1}2^{-k}|\mu_{n,V}(g_{k}) -
E_{\beta_{n}}(\omega_{n})(g_{k})|\leq \epsilon 
\nu_{n}(U)+\nu_{n}(U^{c}).
\] 
We have for all $k\geq 1$,
\[
|\mu_{n,V}(g_{k})-E_{\beta_{n}}(\omega_{n})(g_{k})|\leq A+B+C
\]where,
\[
A=
\frac{\sum_{j=1}^{N'}\int_{(x:\delta_{n}(x)\in U_{j}^{\epsilon})}
e^{n\int \gamma
  d\delta_{n}(x)}|\delta_{n}(x)(g_{k})-m_{j}(g_{k})| l_{n}(dx)} 
{\int_{X}e^{n\int \gamma d\delta_{n}(x)}l_{n}(dx)},
\]
\[
B=\frac{\int_{(x:\delta_{n}(x)\in V\backslash U)}
e^{n\int \gamma
  d\delta_{n}(x)}\delta_{n}(x)(g_{k})l_{n}(dx)} 
{\int_{X}e^{n\int \gamma d\delta_{n}(x)}l_{n}(dx)}
\]
\[
C=|(1-\nu_{n}(U))m_{0}(g_{k})|.
\]
Thus, since we have for all $k\geq 1$, $\|g_{k}\|=1$,  by definition
of $\nu_{n}$ we get,
\begin{eqnarray*}
&&\sum_{k\geq 1}2^{-k}|\mu_{n,V}(g_{k})-E_{\beta_{n}}(\omega_{n})(g_{k})|\\
&\leq& \epsilon
  \sum_{j=1}^{N'}\nu_{n}(U_{j}^{\epsilon})+\frac{1}{2}\nu_{n}(U^{c})+
\frac{1}{2}(1-\nu_{n}(U))\\
&=& \epsilon \nu_{n}(U)+\nu_{n}(U^{c}).
\end{eqnarray*}
Finally we have obtained that
\begin{eqnarray*}
d(\mu_{n},E_{\beta_{n}}(\omega_{n}))&=&
\sum_{k\geq 1}2^{-k}|\mu_{n}(g_{k}) -
E_{\beta_{n}}(\omega_{n})(g_{k})|\\
&\leq& \epsilon \nu_{n}(U)+\frac{3}{2}\nu_{n}(U^{c}). 
\end{eqnarray*}
This implies the desired inequality,
\[
d(\mu_{n},V)\leq \epsilon
\nu_{n}(U)+\frac{3}{2}\nu_{n}(U^{c}). 
\]
By Theorem $1$ $(2)$, since $U^{c}$ is closed, we know that $\lim_{n\rightarrow
  \infty}\nu_{n}(U)=1$. Thus, $\limsup_{n\rightarrow
  \infty}d(\mu_{n},V)\leq \epsilon$, for all $\epsilon >0$. We
  conclude that $\limsup_{n\rightarrow
  \infty}d(\mu_{n},V)=0$. The neighborhood $V$ of
  $\mathcal{P}_{e}(\gamma)$ being arbitrary, this implies that all limit
  measures of $\mu_{n}$ are contained in $\mathcal{P}_{e}(\gamma)$. In
  particular, if $\mathcal{P}_{e}(\gamma)$ is reduced to one measure
  $\mu$, this shows that $\mu_{n}$ converges to $\mu$.
\end{proof}

\subsection{Proof of Theorem $2$}
\subsubsection{Proof of part $(1)$}
\begin{proof}
The map $g \rightarrow \delta_{n}(x)(g)$ being continuous, Theorem $2$
$(1)$ follows from Theorem $1$ $(1)$.
\end{proof}
\subsubsection{Proof of part $(2)$}
\begin{proof}
If $J_{g}(O_{d})= +\infty$ then there is nothing to do. Suppose
then $J_{g}(O_{d})< +\infty$.
Let $\varepsilon>0$ and choose $\alpha_{\varepsilon} \in O_{d}$ with
$\mathcal{P}_{g,\alpha_{\varepsilon}}(X) \ne \emptyset$ such that
\[
J_{g}(O_{d})>I_{g}(\alpha_{\varepsilon})-\varepsilon,
\]
We know from (\cite{Roc} Theorem 23.4 and 23.5) that, given $\alpha$ in
the interior of the affine hull of the domain $D(J_{g})$ of $J_{g}$,
there exists $\beta \in \real^{d}$ such that 
\[
Q_{\gamma}(\beta \cdot g)=\beta \cdot \alpha - J_{g}(\alpha).
\]
Let then $\beta_{\varepsilon}\in \real^{d}$ such that
\begin{equation}
Q_{\gamma}(\beta_{\varepsilon}\cdot g)=\beta_{\varepsilon} \cdot
\alpha_{\varepsilon}- J_{g}(\alpha_{\varepsilon}).
\end{equation}
Consider now a small neighborhood of $\alpha_{\varepsilon}$,
\[
O_{d,r}:=\{ \alpha \in \real^{d}:|\alpha_{\varepsilon}-\alpha|\leq r\},
\]
such that $O_{d,r} \subset O_{d}$.
Set 
\[
\psi_{n}(O_{d}):=\int_{X}e^{n\int \gamma d\delta_{n}(x)}
1_{(\delta_{n}(x)(g) \in O_{d})}l_{n}(dx),
\]
and $Z_{n}(O_{d}):=\frac{\psi_{n}(O_{d})}{\int_{X}e^{n\int \gamma
d\delta_{n}(x)}l_{n}(dx)}$.
We have, $Z_{n}(O_{d})\geq Z_{n}(O_{d,r})$ and
\[
\psi_{n}(O_{d,r})
\geq e^{-n\beta_{\epsilon}\cdot \alpha_{\epsilon}}e^{-n\|\beta_{\epsilon}\|r}
\int_{X}e^{n\int (\gamma+\beta_{\epsilon}\cdot g)d\delta_{n}(x)}
1_{(\delta_{n}(x)(g) \in O_{d,r})}l_{n}(dx)
\]
Thus
\[
 Z_{n}(O_{d,r})\geq e^{-n\beta_{\epsilon}\cdot
 \alpha_{\epsilon}}e^{-n\|\beta_{\epsilon}\|r} 
\frac{\int_{X}e^{n\int (\gamma+\beta_{\epsilon}\cdot g)d\delta_{n}(x)}
1_{(\delta_{n}(x)(g) \in O_{d,r})}l_{n}(dx)}
{\int_{X}e^{n\int \gamma d\delta_{n}(x)}l_{n}(dx)}.
\]
Set
\[
Z_{n}^{\varepsilon}(O_{d,r}):=
\frac{1}{n}\log
\frac{\int_{X}e^{n\int (\gamma+\beta_{\epsilon}\cdot g)d\delta_{n}(x)} 
1_{(\delta_{n}(x)(g) \in O_{d,r})}l_{n}(dx)}
{\int_{X}e^{n\int (\gamma+\beta_{\epsilon}\cdot g)d\delta_{n}(x)}l_{n}(dx)}
\]and
\[
Z_{n}(\beta_{\varepsilon}\cdot g):=\frac{1}{n}\log 
\frac{\int_{X}e^{n\int (\gamma+\beta_{\epsilon}\cdot g)d\delta_{n}(x)}
l_{n}(dx)} 
{\int_{X}e^{n\int \gamma d\delta_{n}(x)}l_{n}(dx)}.
\]
Therefore,
\[
\frac{1}{n}\log Z_{n}(O_{d,r})\geq 
-r\|\beta_{\varepsilon}\|+ 
(Z_{n}(\beta_{\varepsilon}\cdot g)-\beta_{\varepsilon}\cdot
\alpha_{\varepsilon})+ 
\frac{1}{n}\log Z_{n}^{\varepsilon}(O_{d,r}).
\]
From Kozlovski's theorem we deduce that
\[
\lim_{n\rightarrow
  \infty}Z_{n}(\beta_{\varepsilon}\cdot g)=
P(\gamma+\beta_{\varepsilon}\cdot
  g)-P(\gamma)=Q_{\gamma}(\beta_{\varepsilon}\cdot g).
\]
Thus
\[
\liminf_{n\rightarrow \infty}\frac{1}{n}\log Z_{n}(O_{d,r})\geq
-r\|\beta_{\varepsilon}\|+
 (Q_{\gamma}(\beta_{\varepsilon}\cdot g)-\beta_{\varepsilon}\cdot
\alpha_{\varepsilon})+\lim_{n\rightarrow \infty}\frac{1}{n}\log
 Z_{n}^{\varepsilon}(O_{d,r}).
\]
We will show that
\begin{equation}
\lim_{n\rightarrow \infty} Z_{n}^{\varepsilon}(O_{d,r})=1.
\end{equation}
Let us show how to finish the proof using $(11)$ :
\begin{eqnarray*}
\liminf_{n\rightarrow \infty} \frac{1}{n}\log Z_{n}(O_{d})&\geq&
-r\|\beta_{\varepsilon}\|+Q_{\gamma}(\beta_{\varepsilon} \cdot
g)-\beta_{\varepsilon} 
\cdot \alpha_{\varepsilon}\\ 
&=&-r\|\beta_{\varepsilon}\|-J_{g}(\alpha_{\varepsilon}) \\
&\geq& -r\|\beta_{\varepsilon}\|-J_{g}(O_{d}) -\varepsilon, 
\end{eqnarray*}
for any $\varepsilon >0$. Since $r>0$ was arbitray choosen, we let
$r\rightarrow 0$ and $\varepsilon \rightarrow 0$ respectively and we get
$\liminf_{T\rightarrow \infty} \frac{1}{T}\log Z_{T}(O_{d})\geq
-J_{g}(O_{d})$ which completes the proof of Theorem $2$ ($2$). 

It remains to show ($11$). Let $K_{d,r}$ be the complement set of
$O_{d,r}$ in the image 
$g*(\mathcal{P}_{inv}(X))$ of $\mathcal{P}_{inv}(X)$ under the continuous map
$g* : m\rightarrow m(g)$. We have
$Z_{n}^{\varepsilon}(O_{d,r})+Z_{n}^{\varepsilon}(K_{d,r})=1$. The
goal is to show using Theorem $1$ that, $Z_{n}^{\varepsilon}(K_{d,r})$
decrease to zero as $n\rightarrow \infty$ (in fact exponentially fast).

Consider $J^{\varepsilon}:=J_{\gamma+ \beta_{\varepsilon} \cdot g}$,
which is the functional $J$ corresponding to
$Q^{\varepsilon}:=Q_{\gamma+ \beta_{\varepsilon} \cdot g}$.
By definitions $(3)$ and $(4)$, we have for any continuous function
$\omega$ on $X$,
\[
Q^{\varepsilon}(\omega)=P(\gamma+ \beta_{\varepsilon} \cdot g
+\omega)-P(\gamma+ \beta_{\varepsilon} \cdot g),
\]
and for any probability measure $m$ on $X$,
\[
J^{\varepsilon}(m)=\sup_{\omega}(\int \omega dm -Q^{\varepsilon}(\omega)).
\]
From this we deduce easily that
\[
J^{\varepsilon}(m):= J_{\gamma}(m)+Q_{\gamma}(\beta_{\varepsilon}\cdot
g)-\int \beta_{\varepsilon}\cdot g dm,
\]and
\[
\inf_{m(g)=\alpha}J^{\varepsilon}(m)=\inf_{m(g)=\alpha}J_{\gamma}(m)+
Q_{\gamma}(\beta_{\varepsilon}\cdot
g)- \beta_{\varepsilon}\cdot \alpha.
\]
The set $K_{d,r}$ is compact in $\real^{d}$, then by Theorem $1$ $(1)$
we get
\[
\limsup_{n \rightarrow \infty}\frac{1}{n}\log
Z_{n}^{\varepsilon}(K_{d,r})
\leq -J^{\varepsilon}(K),
\]where $K:=(g*)^{-1}(K_{d,r})$ which is a closed subset of $\mathcal{P}(X)$.
If $J^{\varepsilon}(K)=+\infty$ there is nothing to do and the result
follows. The key point is to prove that $J^{\varepsilon}(K)>0$. For
this, set
Set
\[
J_{g}^{\varepsilon}(\alpha):=
J_{g}(\alpha)+Q(\beta_{\varepsilon}\cdot g)-\beta_{\varepsilon}   
\cdot \alpha.
\]
The functional $J^{\varepsilon}$ is non negative (since
$Q^{\varepsilon}(0)=0$), lower 
semicontinuous and then it achieves its minimum on compact sets. 
We have $J_{g}^{\varepsilon}(\alpha)\geq 0$ and 
$J_{g}^{\varepsilon}(\alpha_{\varepsilon})=0$ (see $(10)$). Recall that,
if $J_{g}^{\varepsilon}(\alpha)=0$ for some $\alpha \in K_{d,r}$, then
there will 
correspond to $\alpha$ an equilibrium state $m_{\alpha}\in K$ for the
potential $\gamma+\beta_{\varepsilon}\cdot g$ such that
$m_{\alpha}(g)=\alpha$.  
The vector $\alpha_{\varepsilon}$ is the unique point realizing the
minimum, i.e the unique solution for the
equation $J_{g}^{\varepsilon}(\alpha)=0$.
Indeed, two different solutions will produce two distinct
equilibrium states for the potential $\gamma+\beta_{\varepsilon}\cdot
g$ which contradicts our standing assumption of Theorem $1$.
Since $\alpha_{\varepsilon} \in O_{d,r}$, then
$J_{g}^{\varepsilon}(\alpha)>0$ for $\alpha \in K_{d,r}$. On the other hand
the set $K_{d,r}$ being compact, by the lower semicontinuity of
$J^{\varepsilon}$ we have $J^{\varepsilon}(K)=\inf_{\alpha \in
  K_{d,r}} J_{g}^{\varepsilon}(\alpha)>0$. Thus we have proved that
\[
\limsup_{n\rightarrow \infty} \frac{1}{n}\log
Z_{n}^{\varepsilon}(K_{d,r}) \leq -J^{\varepsilon}(K)<0
\]
from which ($11$) follows immediately.
\end{proof}
\subsection{Proof of Theorem $1$ ($4$)}
\begin{proof}
Let $O \subset \mathcal{P}(X)$ be an open set, $\epsilon >0$ and
choose $m_{\epsilon} \in O$ such that 
\[
I(m_{\epsilon}) \leq \rho(O) +\epsilon.
\]
For each $N$ we set,
\[
d_{N}(m, m'):=\sum_{k=1}^{N}2^{-k}|\int g_{k}dm - \int g_{k}dm' |.
\] 
Set $2r=\inf \{d(m, m_{\epsilon}): m\in \mathcal{P}(X)\backslash
O\}$. We have $r>0$, since $\mathcal{P}(X)\backslash
O$ is a compact subset of $\mathcal{P}(X)$.
Since for all $k$, $\|g_{k}\|=1$, we have $0\leq d(m,
m')-d_{N}(m, m')\leq 2^{-(N-1)}$. Thus, for $N$ suffuciently
large,
\[
O_{\epsilon, r}:=\{m\in \mathcal{P}(X):d_{N}(m, m_{\epsilon})<r\}
\subset O. 
\]
For each $\alpha =(\alpha_{1}, \ldots, \alpha_{N})\in \real^{N}$
denote $\|\alpha\|_{N}=\sum_{k=1}^{N}2^{-k}|\alpha_{k}|$.
Let $g=(g_{1}, \ldots, g_{N})$ and set
$\alpha_{\epsilon}= m_{\epsilon}(g):=(\int
g_{1}dm_{\epsilon},\ldots, \int g_{N}dm_{\epsilon})$ and
\[
O_{N,r}:=\{\alpha \in \real^{N}: \|\alpha_{\epsilon}-\alpha \|_{N}<r\}.
\]
Then, $g (O_{\epsilon, r})=O_{N,r}\cap g(\mathcal{P}(X))$. From
Theorem $2$ ($2$) we get,  
\begin{eqnarray*}
&& \liminf_{n\rightarrow \infty}\frac{1}{n}\log
\frac{\int_{X}e^{n\int \gamma d \delta_{n}(x)}
1_{(\delta_{n}(x) \in O)}l_{n}(dx)}
{\int_{X}e^{n\int \gamma d \delta_{n}(x)}l_{n}(dx)}\\
&\geq& \liminf_{n\rightarrow \infty}\frac{1}{n}\log
\frac{\int_{X}e^{n\int \gamma d \delta_{n}(x)}
 1_{(\delta_{n}(x) \in O_{\epsilon, r})}l_{n}(dx)}
{\int_{X}e^{n\int \gamma d \delta_{n}(x)}l_{n}(dx)}\\
&=& \liminf_{n\rightarrow \infty}\frac{1}{n}\log
\frac{\int_{X}e^{n\int \gamma d \delta_{n}(x)}
1_{(\delta_{n}(x)(g) \in O_{N,r})}l_{n}(dx)}
{\int_{X}e^{n\int \gamma d \delta_{n}(x)}l_{n}(dx)}\\
&\geq& -J_{g}(O_{N, r})\\
&\geq& -J_{g}(\alpha_{\epsilon})
\geq -J(m_{\epsilon}) \geq -J(O) -\epsilon,
\end{eqnarray*}
for any $\epsilon >0$. This complete the proof of Theorem $1$ $(4)$.
\end{proof}

\end{document}